\documentclass[english]{amsart}
\usepackage{amsmath}
\usepackage{amssymb}
\usepackage[all]{xy}
\usepackage{xypic,amsthm,amssymb,hyperref,amsmath,graphicx,stmaryrd,boxedminipage,mathrsfs,manfnt,color,setspace,comment}
\usepackage[all]{xy}
\usepackage[mathcal]{euscript}
\usepackage{calligra}
\usepackage{sseq}

\newcommand{\sma}{\wedge}

\newcommand{\Sp}{\operatorname{Sp}}

\renewcommand{\hom}{\operatorname{Hom}}

\newcommand{\id}{\operatorname{Id}}

\newcommand{\map}{\operatorname{Map}}

\DeclareMathOperator*{\colim}{colim}

\newcommand{\mc}{\mathcal}

\newcommand{\mbf}{\mathbf}
\newcommand{\thh}{\operatorname{THH}}

\newcommand{\ho}{\operatorname{Ho}}

\newcommand{\cat}{\mc{C}\text{at}}
\newcommand{\sset}{\mc{S}\text{et}_\Delta}

\theoremstyle{plain}
\newtheorem{thm}{Theorem}[section]
\newtheorem{lem}[thm]{Lemma}
\newtheorem{cor}[thm]{Corollary}
\newtheorem{prop}[thm]{Proposition}

\theoremstyle{definition}
\newtheorem{defn}[thm]{Definition}

\newtheorem{rmk}[thm]{Remark}

\newtheorem{nota}[thm]{Notation}

\title{(Co)Homology of Spectral Categories}
\author{Jonathan A. Campbell}
\date{}

\begin{document}

\maketitle

\begin{abstract}
In this article we develop the cotangent complex and (co)homology theories for spectral categories. Along the way, we reproduce standard model structures on spectral categories. As applications, we show that the invariants to descend to stable $\infty$-categories and we prove a stabilization result for spectral categories. 
\end{abstract}

\tableofcontents

\section{Introduction}

Let $\mc{C}$ be a category with ``enough'' limits. Under sufficiently nice conditions, one may speak of the abelian objects in $\mc{C}$ as those $C \in \mc{C}$ equipped with multiplication $C \times C \to C$, unit $\ast \to C$ and identity maps $i: C \to C$ that satisfy some reasonable axioms. The full subcategory of these objects will be denoted $\mc{C}_{\text{ab}}$ and there is then a natural forgetful functor $\mc{C}_{ab} \to \mc{C}$ the left adjoint of which is called abelianization. A classical example of such a construction is the infinite symmetric product $\operatorname{SP}(X)$ of a CW-complex $X$. By the Dold-Thom theorem, this in fact models homology: $\pi_\ast \operatorname{SP} (X) \cong H_\ast (X)$. In fact, in general when one has model structures on $\mc{C}$ and $\mc{C}_{ab}$, the left derived functor of abelianization models some kind of homology theory. This is the guiding principle in the discussion below. 

Another classical example is that of $R$-algebras. Let $R$ be a commutative ring and $A$ an $R$-algebra. Then, the abelianization of the category of $R$-algebras over $A$, denoted $(\mbf{Alg}_{R})_{/A}$, is the category of $A$-modules. The abelianization of $(B \to A)$ is the module $B \otimes_A \Omega_{B/A}$ where $\Omega_{B/A}$ denote the usual K\"{a}hler differentials. In this situation there are no model structures, and thus no homology theories, but the situation changes when considering simplicial commutative rings.  In \cite{quillen} Quillen develops the (co)homology of commutative rings from this perspective (the same homology theory was independently developed by Andre\cite{andre}). Quillen computes the abelianization of the category of simplicial $R$-algebras over $A$. The corresponding cohomology theory is corepresented by an object called the cotangent complex, which is the correct derived version of K\"{a}hler differentials. The cotangent complex is closely related to deformation theory and is central in algebraic geometry \cite{illusie_1,illusie_2,lichtenbaum_schlessinger}.

Commutative rings are, however, the shadow of another derived object, commutative or $E_\infty$, ring spectra \cite{EKMM}. These may be viewed as the proper topological analogues of commutative rings, and one may ask if they have analogues of various algebraic notions. The question naturally arises of what the Andre-Quillen homology of a commutative ring spectrum should be. In \cite{basterra}, Basterra answers this question and in \cite{basterra_mandell}, Basterra and Mandell answer the question for ring spectra over operads other than the $E_\infty$-operad and compute the contangent complex in some examples.

Another way to derive rings is to first embed them into the category of schemes and then embed schemes, via their derived categories, into differential graded categories \cite{keller}. In this case, one could ask that the corresponding dg-categories have the full complement of notions and invariants that rings carry. This is the crux of Kontsevich's approach to derived algebraic geometry. In this case, what should the analgoue of the cotangent complex be? The question is answered in \cite{tabuada}, where Tabuada develops a notion of ``non-commutative Andre-Quillen homology''. However, dg-categories embed into spectral categories, and we could ask our usual questions. This paper answers these questions. We provide a cohomology theory for spectral categories, and relate it to other known cohomology theories of spectral categories.

Before arriving at a definition of a cohomology theory for spectral categories, we must set down the appropriate model category foundations. In order to defined the derived left adjoint of abelianization, we give a model structure on spectral categories (with fixed objects, see Section 2):

\begin{thm}
Let $\mc{C}\text{at}^{Sp}_{\mc{O}}$ be the category of spectral categories with fixed object set $\mc{O}$. Then $\mc{C}\text{at}^{Sp}_{\mc{O}}$ is a simplicial cofibrantly generated model category. 
\end{thm}

Similar results are obtained in \cite{schwede_shipley_monoidal} and \cite{dundas_V_categories}. We reprove the theorem in order to use the techniques for producing model structures on more exotic types of spectral categories.

After this is set up, we follow Basterra in producing a string of adjunctions which computes the left adjoint of abelianzation and defines the cotangent complex. With a contangent complex in hand, we have 

\begin{thm}
Let $\mc{A}$ be a spectral category. Then there is a cohomology theory $\operatorname{TDC}$  on spectral categories $\mc{C}\text{at}^{\text{Sp}}_{/\mc{A}}$. There is also a homology theory,  $\operatorname{TDH}$. 
\end{thm}

\begin{rmk}
$\operatorname{TDC}$ stands for ``topological derivation cohomology'' and similarly for ``homology.'' We avoid using ``Andre-Quillen'' homology since that usually refers to commutative contexts. 
\end{rmk}

These homology and cohomology theories fit into exact sequences with well known homology theories on categories. In particular, it is related to topological Hochschild homology (hereafter, $\thh$)

\begin{thm}
Let $\mc{C}$ be a spectral category. Then there is a cofiber sequence of spectra
\[
\operatorname{TDH}(\mc{C}) \to \bigvee_{c,d \in \operatorname{ob} \mc{C}} \mc{C}(c,d) \to \thh (\mc{C}) 
\]
\end{thm}

Once we have set up cohomology theories, we move on to discuss stabilization. As in the case for rings and spectra \cite{basterra_mandell}, once can identify the stabilization of spectral categories with some category of modules. More precisely, we have the following theorem. 

\begin{thm}
Let $\mc{A}$ be a spectral category, then there is a zig-zag of Quillen equivalences
\[
\operatorname{Mod}_{(\mc{A},\mc{A})} \leftarrow \cdots \rightarrow \operatorname{Sp}(\cat^{(\mc{A},\mc{A})\text{-act}}_{/\mc{A}}). 
\]
Furthermore, under this equivalence, an $\mc{A}$-bimodule $\mc{M}$ is sent to the spectrum that corepresents the cotangent complex. 
\end{thm}

This is useful, for example, in placing the results of Dundas-McCarthy \cite{dundas_mccarthy} on the derivative of K-theory into a more standard context. We hope to return to this in future work.

For the convenience of the reader we provide a rough table of analogies

\vspace{2mm}

\begin{center}
\begin{tabular}{|c|c|}\hline
\qquad ring \qquad & spectral category\\\hline
$A$ & $\mc{A}$\\ \hline
$(\operatorname{CAlg}_A)_{/A}$ & $\cat^{(\mc{A},\mc{A})\text{act}}_{/\mc{A}}$\\\hline
$\Omega_A$ & $\mbf{L}_{\mc{A}}$ \\\hline
$\Omega_{B/A}$ & $\mathbf{L}_{\mc{B}/\mc{A}}$ \\\hline
$AQ$-homology & $\operatorname{TDH}$ \\\hline
$AQ$-cohomology & $\operatorname{TDC}$ \\\hline
\end{tabular}
\end{center}

\section{Preliminary Remarks}

Below, we work with categories with fixed objects. Some preliminary remarks are in order to clarify why we made this choice. In fact, this is not necessary, but will make our work significantly easier. Since we seek to define some sort of cotangent complex, what we really want to do is construct a left adjoint to an square-zero extension functor. Let us briefly explain square zero extensions with elaboration left for the sequel. The contents of this section are not strictly necessary for the discussion below and may be safely skipped by the reader. 

Let $\mc{A}$ be a spectral category and $\mc{M}: \mc{A} \sma \mc{A}^{\text{op}} \to \operatorname{Sp}$ an $\mc{A}$-bimodule (Defn. \ref{bimodule}). Then we may form a spectral category $\mc{A} \ltimes \mc{M}$ as follows. Let $\mc{A} \ltimes \mc{M} (x, y) = \mc{A}(x,y) \vee \mc{M}(x,y)$ and define composition via composition of $\mc{A}$, the module action of $\mc{A}$ on $\mc{M}$ and the double $\mc{M}$ terms as zero. A priori this defines a functor
\[
\mc{A} \ltimes (-) : \operatorname{Mod}_{(\mc{A},\mc{A})} \to \cat_{\operatorname{Sp}}. 
\]
Since we'll be interested in stabilization, we should specify that in fact we have a functor
\[
\mc{A} \ltimes (-): \operatorname{Mod}_{(\mc{A},\mc{A})} \to (\cat_{\operatorname{Sp}})_{/\mc{A}}. 
\]
By Freyd's adjoint functor theorem, using an argument analagous to the one found in \cite{tabuada}, this functor has a left adjoint. Further, it is not hard to show that the left adjoint respects various homotopical structures and can be derived. So we have a functor, $\Omega^1$, such that
\[
\hom_{\cat_{/\mc{A}}} (\mc{C}, \mc{A} \ltimes \mc{M}) \cong \hom_{\operatorname{Mod}_{(\mc{A},\mc{A})}} (\Omega^1 \mc{C}, \mc{M}). 
\]

However, in this form the functor $\Omega^1$ is not easy to identify explicitly. It is much more productive to work with categories with fixed objects. The following discussion illustrates that we lose no essential information by doing so. We begin with functor that brutally reduces spectral categories over $\mc{A}$ to have the same object set as $\mc{A}$. 

\begin{defn}
Let $\mc{A}$ be a spectral category with object set $|\mc{A}|$. Let $\operatorname{Red}$ be a functor (``reduction'')
\[
\operatorname{Red}: \cat^{\Sp}_{/\mc{A}} \to \cat^{|\mc{A}|, \Sp}_{/\mc{A}} 
\]
defined as follows. Given $F: \mc{C} \to \mc{A}$ we define a new category $\widetilde{\mc{C}} = \operatorname{Red}(\mc{C})$ with
\begin{itemize}
\item[-] \textbf{objects}: the same objects as $\mc{A}$. 
\item[-] \textbf{morphisms}: Given $a, b \in \mc{A}$ consider the sets $\{c_i\}_{i \in I}$, $\{d_j\}_{j \in J }$ which are the pre-images under $F$ of $a$ and $b$, respectively. Define the morphisms to be
\[
\widetilde{\mc{C}}(a, b) := \bigvee_{I, J} \mc{C}(c_i, d_j)
\]
\item[-] \textbf{composition}: The composition
\[
\widetilde{\mc{C}}(a, b) \sma \widetilde{\mc{C}}(b, c) = \left(\bigvee_{I,J} \mc{C}(c_i, d_j)\right) \sma \left(\bigvee_{J, K} \mc{C}(d_j, e_k)\right) \to \widetilde{\mc{C}}(a,c)
\]
is given by compositions and mapping non-composable smash products to a point. 
\end{itemize}
\end{defn}

Let us check that this is a functor. Let $F': \mc{D} \to \mc{A}$ be another category over $\mc{A}$ and $G: \mc{C} \to \mc{D}$ be a functor over $\mc{A}$, so that $F = (G \circ F')$. Now $G$ induces a map $\widetilde{\mc{C}}(a,b)\to \widetilde{\mc{D}}(a,b)$ which is 
\[
\bigvee_{I,J} \mc{C}(c_i, d_j) \xrightarrow{G} \bigvee_{I,J} \mc{D}(G(c_i), G(d_j)) \hookrightarrow \bigvee_{I',J'} \mc{D}(c'_i, d'_j)
\]
where $F^{-1}(a) = \{c_i\}_I, F^{-1}(b) = \{d_j\}_J$ and $(F')^{-1}(a) = \{c'_i\}_{I'}, (F')^{-1}(b) = \{d'_j\}_{J'}$. Composition is respected, and thus $G$ induces a map $\widetilde{\mc{C}} \to \widetilde{\mc{D}}$. 

\begin{prop}
The functor $\operatorname{Red}$ defines a left adjoint to the fully faithful inclusion $\iota: \cat^{|\mc{A}|}_{/\mc{A}} \to \cat_{/\mc{A}}$. That is, we have an adjunction
\[
\operatorname{Red}:  \cat_{/\mc{A}} \leftrightarrows  \cat^{|\mc{A}|}_{/\mc{A}} : \iota 
\]
\end{prop}

\begin{proof}
To see that it is an adjunction we have to check that
\[
\hom_{\cat^{|\mc{A}|}_{/\mc{A}}} (\operatorname{Red}(\mc{B}), \mc{C}) \xrightarrow{\cong} \hom_{\cat_{/\mc{A}}} (\mc{B}, \iota \mc{C}).
\]
For $a, b \in \operatorname{ob}(\mc{A})$, the set on the left is the set of spectrum maps $\bigvee \mc{B}(c_i, d_j) \to \mc{C}(a, b)$ compatible with composition, which is of course just equivalent to collections of maps $\mc{B}(c_i, d_j) \to \mc{C}(a, b)$ compatible with composition, which is the hom set on the right. 
\end{proof}

In view of this, we may restrict our attention to categories with fixed objects. Indeed, we have the string of adjunctions
\[
\xymatrix{
\cat^{Sp,|\mc{A}|}_{/\mc{A}} \ar@<.5ex>[r]^{\iota} & \cat^{\Sp}_{/\mc{A}} \ar@<.5ex>[l]^{\operatorname{Red}} \ar@<.5ex>[r]^{\Omega} & \operatorname{Mod}_{(\mc{A},\mc{A})} \ar@<.5ex>[l]^{\mc{A}\ltimes (-)} 
}
\]
Thus, we have the following series of equivalences
\begin{align*}
\hom_{\operatorname{Mod}_{(\mc{A},\mc{A})}} (\Omega^1 \mc{C}, \mc{M}) &\simeq \hom_{\cat^{\Sp}_{/\mc{A}}} (\mc{C}, \mc{A} \ltimes \mc{M}) \\ &\simeq \hom_{\cat^{\Sp,|\mc{A}|}_{/\mc{A}}} (\operatorname{Red}(\mc{C}), \mc{A}\ltimes \mc{M}) \\ 
&\simeq \hom_{\cat^{\Sp}_{/\mc{A}}} (\iota \operatorname{Red}(\mc{C}), \mc{A}\ltimes \mc{M})\\
&\simeq \hom_{\operatorname{Mod}_{(\mc{A},\mc{A})}} (\Omega^1 \operatorname{Red}(\mc{C}), \mc{M})
\end{align*}

where the third equivalence is by full-faithfulness of the inclusion $\iota$. Thus, if we compute a left adjoint to $\mc{A}\ltimes (-)$ out of $\cat^{\Sp,|\mc{A}|}_{/\mc{A}}$, we have computed one out of $\cat^{\Sp}_{/\mc{A}}$. 

In order to explictly identify the left adjoint of the square zero extension it will be considerly easier to work with fixed objects. In particular, once we have done this, there is a strong analogy with the classical one object case and the theory proceeds in much the same way. 

Before we move on, this is a natural place to record the following fact (for definitions, see below)

\begin{prop}
If $\mc{B}, \mc{C}$ are in $\operatorname{Cat}^{\text{Sp}}_{/\mc{A}}$ and there is a DK-equivalence $\mc{B} \xrightarrow{\sim} \mc{C}$, then $\operatorname{Red}(\mc{B}) \to \operatorname{Red}(\mc{C})$ is a DK-equivalence. 
\end{prop}

\begin{rmk}
It is clear that $\operatorname{Red}$ preserves, but does not reflect, DK-equivalences. 
\end{rmk}

\section{The Model Structure on $\cat^{\Sp}_{\mc{O}}$}

The category of spectral categories with fixed objects is equipped with a model structure (Schwede-Shipley) \cite{schwede_shipley_equivalences}. Similar results are obtained by Dundas in \cite{dundas_V_categories}.  The details of the constructions, as well as some elaborations, will be needed later. Thus, we will re-do the construction of that model structure, and show how it arises from standard monadic results collected in \cite{EKMM}. We need some preliminary definitions. 

\begin{defn}
Let $\mc{V}$ be a symmetric monoidal category with unit $\mbf{I}$ and let $\mc{O}$ be any set.  The \textbf{category of $\mc{V}$-graphs} denoted $\mbf{Gr}^{\mc{V}}_{\mc{O}}$
\[
\mbf{Gr}^{\mc{V}}_{\mc{O}} = \prod_{\mc{O} \times \mc{O}} \mc{V}. 
\]
\end{defn}

\begin{rmk}
We think of such product category as a graph having vertices indexed by  $\mc{O}$ and edges indexed by $\mc{O}\times \mc{O}$, each with label from $\mc{V}$. 
\end{rmk}

Roughly,this category is as well behaved as $\mc{V}$:

\begin{prop}
If $\mc{V}$ is complete and cocomplete, then so is $\mbf{Gr}^{\mc{V}}$. If $\mc{V}$ is tensored and cotensored, then so $\mbf{Gr}^{\mc{V}}_{\mc{O}}$. 
\end{prop}
\begin{proof}
All constructions in $\mbf{Gr}^{\mc{V}}_{\mc{O}}$ are defined pointwise. 
\end{proof}

If $\mc{V}$ is equipped with a cofibrantly generated model structure with (acyclic) generating cofibrations $I$ (resp. $J$), then $\mbf{Gr}^{\mc{V}}_{\mc{O}}$ has such a structure as well.

\begin{prop}[\cite{schwede_shipley_equivalences}, p.325]
Let $\mc{V}$ be a cofibrantly generated model category with generating cofibrations $I$ and generating acyclic cofibrations $J$. Then $\mbf{Gr}^{\mc{V}}_{\mc{O}}$ is a cofibrantly generated model category with
\begin{enumerate}
\item generating cofibrations $X_{i,j} \to Y_{i, j}$ where $X_{i,j}$ is the graph where the only non-trivial label is between $i$ and $j$, and $X \to Y$ is a generating cofibration in $\mc{V}$
\item generating acyclic cofibrations exactly as above, but with $X \to Y$ in $J$
\item weak equivalences defined point-wise. 
\end{enumerate}
\end{prop}

A category is in some sense a graph with composition. We may codify such a statement using a monadic description. To this end, we first note that there is extra structure on the category of graphs. 

\begin{defn}[Tensor product on $\mbf{Gr}^{\mc{V}}_{\mc{O}}$]\label{graph_product}
The category of graphs is equipped with a monoidal product (which is, however, not symmetric monoidal). Given $\mc{G}, \mc{H} \in \mbf{Gr}^{\mc{V}}_{\mc{O}}$ we define
\[
(\mc{G} \otimes \mc{H})(i,j) = \bigvee_{k} \mc{G}(i, k) \otimes_{\mc{V}} \mc{G}(k, j)
\]
The unit for the monoidal structure is the category denoted $\mbf{I}_{\mc{O}}$ with
\[
\mbf{I}_{\mc{O}}(i, j) = 
\begin{cases}
 \ast & i \neq j\\
 \mbf{I} & i = j
\end{cases}
\]
\end{defn}

We can now define the ``free category'' monad. We'll use notation that conforms with the notation in \cite{EKMM} for monoids, since an $\mc{V}$-category with fixed objects will end up being a monoid in $\mbf{Gr}^{\mc{V}}_{\mc{O}}$. 

\begin{defn}
For $\mc{G} \in \mbf{Gr}^{\mc{V}}_{\mc{O}}$ we define the \textbf{free category functor} as
\[
\mbf{T}\mc{G} = \bigvee_{n = 0} \mc{G}^{\otimes n}
\]
where we stipulate that $\mc{G}^{0} = \mbf{I}$. We note that this is a monad via a product $\mbf{T} \mbf{T} \to \mbf{T}$ given by concatenation and obvious unit $\mbf{I} \to \mbf{T}$. 
\end{defn}

We may now define $\mc{V}$-categories with fixed objects. 

\begin{defn}
A \textbf{$\mc{V}$-category with fixed objects} is an $\mbf{T}$-algebra in $\mbf{Gr}^{\mc{V}}_{\mc{O}}$. The category of all such will be denoted $\cat^{\mc{V}}_{\mc{O}} := \mbf{Gr}^{\mc{V}}_{\mc{O}}[\mbf{T}]$.  
\end{defn}

We will need some structural results about this category: the existence of limits and colimits, how to compute them, and how to compute tensors if $\mc{V}$ is tensored. We review some theorems from \cite{EKMM}, and show that all structure on $\cat^{\mc{V}}_{\mc{O}}$ follows formally from the fact that it is a category of algebras. For categories of algebras, the following proposition of Hopkins is very useful:

\begin{prop}\cite[II.7.4]{EKMM}\label{monad_cocomplete}
$\mc{C}$ be a cocomplete category and let $\mbf{S}: \mc{C} \to \mc{C}$ be a monad. If $\mbf{S}$ preserves reflexive coequalizers, then the category of algebras $\mc{C}[\mbf{S}]$ is cocomplete. 
\end{prop}

\begin{rmk}
We recall that a reflexive coequalizer is a coequalizer diagram
\[
\xymatrix{
A \ar@<.5ex>[r]^f \ar@<-.5ex>[r]_g & B \ar[r] & C 
}
\]
such that there is a section $s: B \to A$ with $f \circ s = \id$ and $g \circ s = \id$. 
\end{rmk}

The proof is that the colimits in $\mc{C}[\mbf{S}]$ can be computed as a coequlizer. Let $D: I\to \mc{C}[\mbf{S}]$ be a diagram then, $X: = \colim_I D(i)$ is computed the coequalizer 
\[
\xymatrix{
\mbf{S} (\colim_I \mbf{S} D(i)) \ar@<.5ex>[r]\ar@<-.5ex>[r] & \mbf{S} (\colim_I D(i))\ar[r]  & X
}
\]
inside of $\mc{C}$. The authors of \cite{EKMM} also prove the following

\begin{prop}\cite[II.7.2.]{EKMM}
Let $\mc{C}$ be a cocomplete closed weak symmetric monoidal category. Then the monad $\mbf{T}$ in $\mc{C}$ preserves reflexive coequalizers. 
\end{prop}

This allows them to prove structural results about the category of monoids in $\mc{C}$. In our case, however, $\cat^{\mc{V}}_{\mc{O}}$ is not \textit{symmetric} monoidal, just monoidal. Nor have we proven that $\cat^{\mc{V}}_{\mc{O}}$ is closed. Luckily, it turns out than an examination of the proof of the proposition reveals that neither of these conditions are strictly required: it will be enough to show that the functor $- \otimes \mc{G}$ preserves colimits and symmetry is never used in the proof of that fact. 

\begin{lem}
Suppose the tensor product $\otimes_{\mc{V}}$ commutes with colimits. Let $\mc{G} \in \mbf{Gr}^{\mc{V}}_{\mc{O}}$. Then $(-) \otimes \mc{G}: \mbf{Gr}^{\mc{V}}_{\mc{O}} \to \mbf{Gr}^{\mc{V}}_{\mc{O}}$ preserves colimits. 
\end{lem}
\begin{proof}
Let $\mc{H}: D \to \mbf{Gr}^{\mc{V}}_{\mc{O}}$ be a diagram. Colimits are computed pointwise, so we need only check that this holds pointwise. 
\begin{align*}
\left((\colim_D \mc{H}(d)) \otimes \mc{G}\right) (i, j) &= \bigvee_{k \in \mc{O}} (\colim_D \mc{H}(d))(i,k) \otimes_{\mc{V}} \mc{G} (k, j)\\
&=\bigvee_{k \in \mc{O}} \colim_D (\mc{H}(d)(i,j) \otimes_{\mc{V}} \mc{G}(k, j))\\
&= \colim_D \left( \bigvee_k \mc{H}(d)(i, j) \otimes_{\mc{V}} \mc{G}(k, j)\right)\\
& =\colim_D (\mc{H}(d) \otimes \mc{G})
\end{align*}
\end{proof}

\begin{rmk}
In the only case of interest, when $\mc{V}$ is some well-behaved category of spectra, e.g. \cite{EKMM, hss}, the tensor product will commute with colimits. 
\end{rmk}

\begin{rmk}
Of course, this proves that $(-)\otimes\mc{G}$ is a left adjoint. We will not need the right adjoint, so we create no notation for it. 
\end{rmk}

We now state the version of \cite[II.7.2]{EKMM} that we shall need. 

\begin{prop}
Let the symmetric monoidal product $\mc{V}$ preserve colimits. Then the free category monad $\mbf{T}: \mbf{Gr}^{\mc{V}}_{\mc{O}} \to \mbf{Gr}^{\mc{V}}_{\mc{O}}$ preserves reflexive coequalizers. 
\end{prop}
\begin{proof}
Exactly as in \cite{EKMM}. 
\end{proof}

This entitles us to the following. 

\begin{prop}
The $\cat^{\Sp}_{\mc{O}}$ is complete and cocomplete. 
\end{prop}
\begin{proof}
It is cocomplete by Prop. \ref{monad_cocomplete}. It is complete because there is a forgetful functor down to $\mbf{Gr}^{\Sp}_{\mc{O}}$ which is a right adjoint --- all limits are created in $\mbf{Gr}^{\Sp}_{\mc{O}}$ which is complete. 
\end{proof}

We would like to use this monadic structure to induce model structures on $\cat^{\Sp}_{\mc{O}}$ and we would furthermore like them to be simplicial or topological model structures. 

\subsection{Spectrally Enriched Categories}

For this section we assume that the category of spectra we work in is symmetric spectra \cite{hss}, and we denote the category of all such by $\Sp^{\Sigma}$. Thus, our category of spectral categories will be $\cat^{\Sp^{\Sigma}}_{\mc{O}}$. This category is tensored an cotensored over $\sset$ via prolongation from the $\mc{S}\text{et}_{\Delta,\ast}$-functors 
\[
(-)^K, \ (-) \sma K : \mc{S}\text{et}_{\Delta,\ast} \to \mc{S}\text{et}_{\Delta,\ast}
\]

We use the following theorem.

\begin{thm}\cite[2.3]{schwede_shipley_monoidal}
Let $\mc{C}$ be a cofibrantly generated model category with generating (acyclic) cofibrations $I$ (resp. $J$). Suppose that the underlying functor of $\mbf{T}$ commutes with filtered colimits. Suppose further that
\begin{itemize}
\item[(*)] Every relative $\mbf{T}J$-cell complex is a weak equivalence
\end{itemize}
Then the category $\mc{C}[\mbf{T}]$ is a cofibrantly generated model category with generating (acyclic) cofibrations $\mbf{T} I$ (resp. $\mbf{T}J$)
\end{thm}

In \cite{schwede_shipley_monoidal} the hardest part in inducing a model structure on $\mc{C}[\mbf{T}]$ is showing that ($\ast$) holds. It is noted in \cite[6.3]{schwede_shipley_equivalences} that the proof of the monoid axiom can be done pointwise, so $\mc{C}\text{at}^{\Sp}_{\mc{O}}$ does in fact inherit a model structure. 

We have to prove that $\cat^{\Sp}_{\mc{O}}$ is a simplicial model category. Recall the definition:

\begin{defn}[Quillen SM7] \label{SM7}
A model category $\mc{C}$ is a \textbf{simplcial model category} if $\mc{C}$ is tensored and cotensored over simplicial sets and for every cofibration $i: A \to B$ in $\mc{C}$ and fibrations $p: X \to Y$ in $\mc{C}$ the map
\[
\xymatrix{
\map(B, X) \ar[r] & \map(A, X) \times_{\map(A, Y)} \map(B, Y)
}
\]
is a fibration and is trivial if either $i$ or $p$ is a weak equivalence. 
\end{defn}

In order to prove that $\cat^{\Sp^{\Sigma}}_{\mc{O}}$ is a simplicial model category let us first prove that  $\mbf{Gr}^{\Sp^{\Sigma}}_{\mc{O}}$ is. 

\begin{prop}
$\mbf{Gr}^{\Sp^{\Sigma}}_{\mc{O}}$ is a cofibrantly generated simplicial model category. 
\end{prop}
\begin{proof}
It suffices to show that the maps under consideration are fibrations when $i: A \to B$ is a generating cofibration and trivial fibrations when $i: A \to B$ is a generating acyclic cofibration. However, the generating cofibrations have the form $A_{ij} \to B_{ij}$ and we note that since the maps are defined pointwise we have
\[
\map_{\mbf{Gr}} (A_{ij} , X) = \map_{\Sp^{\Sigma}} (A, X_{ij}).
\]
The result then follows from the fact that $\Sp^{\Sigma}$ satisfies Quillen's SM7. 
\end{proof}

We now need to construct tensors and cotensors of $\cat^{\Sp^\Sigma}_{\mc{O}}$ by $\sset$. This is done in the standard monadic way, which we will review. The author learned this material from \cite{EKMM}, in particular in the proof of \cite[VII.2.10]{EKMM}. In order to properly define the tensors and cotensors in a category of algebras, we need to define one map which will be useful. Also, for as long as possible we will work with general monads and (co)tensored categories so as to expose what part of the arguments are formal. We maintain the notation of \cite{EKMM} except for the fact that we use $\mbf{M}$ for a general monad instead of $\mbf{T}$, so as to avoid confusion with the free category monad above. 

\begin{defn}\cite[VII.2.10]{EKMM}
  Let $\mc{C}$ be a category tensored over simplicial sets. Let
  \[a: X \to \mc{C}(C, C \otimes X)\]
  be adjoint to the identity $C \otimes X \to C \otimes X$. 

We define the map $\nu: \mbf{M} C \otimes X \to \mbf{M}(C \otimes X)$ to be adjoint to 
\[
\xymatrix{
X \ar[r]^<<<<<{a} & \mc{C}(C, C \otimes X) \ar[r]^<<<<<{\mbf{M}} & \mc{C}(\mbf{M} C, \mbf{M}(C \otimes X))
}
\]
\end{defn}

Let $C, C'$ be $\mbf{M}$-algebras with structure maps $\xi: \mbf{M}C \to C$ and $\xi': \mbf{M}C' \to C'$. In category $\mc{C}$ with equalizers, we can compute the mapping spaces as the equalizer displayed below
\[
\xymatrix{
\mc{C}[\mbf{M}] (C, C') \ar@{.>}[r] & \mc{C}(C,C') \ar@/_/[rr]\ar@<.5ex>[r] & \mc{C}(\mbf{M}C, \mbf{M} C') \ar@<.5ex>[r] & \mc{C} (\mbf{M} C, C')
}
\]

We can also define the tensor $C \otimes X$ in $\mc{C}[\mbf{M}]$ (as opposed to just $\mc{C}$) to be the displayed coequalizer
\[
\xymatrix{
\mbf{M}(\mbf{M} C \otimes X) \ar@<.5ex>[r]^{\mbf{M}(\xi \otimes \id)} \ar@<-.5ex>[r]_{\mu \circ \mbf{M} \nu} & \mbf{M}(C \otimes X) \ar@{.>}[r] & C \otimes X
}
\]

We note that since $\cat^{\Sp^{\Sigma}}_{\mc{O}}$ is complete and cocomplete, and thus has equalizers and coequalizers, we immediately obtain the following. 

\begin{lem}
$\cat^{\Sp^{\Sigma}}_{\mc{O}}$ is tensored and cotensored over $\mc{S}\text{et}_{\Delta}$ and $\mc{S}\text{et}_{\Delta,\ast}$. Furthemore, it is enriched over $\mc{S}\text{et}_{\Delta}$. 
\end{lem}
\begin{proof}
$\cat^{\Sp^{\Sigma}}_{\mc{O}}$ is the category of $\mbf{T}$-algebras in $\mbf{Gr}^{\Sp^{\Sigma}}_{\mc{O}}$. It is complete and cocomplete, and so tensors and cotensors can be constructed as above. Similarly, the mapping space can be computed as above. 
\end{proof}

The enrichment is compatible with the model structure:

\begin{thm}
$\mc{C}\text{at}^{\Sp^{\Sigma}}_{\mc{O}}$ is a simplicial cofibrantly generated model category. 
\end{thm}

\begin{proof}
We have to show that for a cofibration $i: \mc{A} \to \mc{B}$ in $\cat^{\Sp^{\Sigma}}_{\mc{O}}$ and and fibration $p: \mc{X} \to \mc{Y}$ of in $\cat^{\Sp^{\Sigma}}_{\mc{O}}$ that the map
\[
\map(\mc{B},\mc{X}) \to \map (\mc{A},\mc{X}) \times_{\map(\mc{A},\mc{X})} \map(\mc{B},\mc{Y})
\]
is a fibration of simplicial sets. It is enough to show this when $i: \mc{A} \to \mc{B}$ is a map of the form $\mbf{T} A_{ij} \to \mbf{T} B_{ij}$ where $A \to B$ is a a cofibration in $\Sp^{\Sigma}$. We thus check the right lifting property with respect to $\Lambda^n_i \to \Delta^n$ with $0 \leq i \leq n$. Since
\[
\map_{\cat^{\Sp^{\Sigma}}_{\mc{O}}} (\mbf{T} X, \mc{X}) \cong \map_{\mbf{Gr}^{\Sp^{\Sigma}}_{\mc{O}}} (X, \mc{X})
\]
by the free-forgetful adjunction, we have the the map above is equivalent to 
\[
\map_{\mbf{Gr}} (B_{ij}, \mc{X}) \to \map_{\mbf{Gr}} (A_{ij}, \mc{X}) \times_{\map_{\mbf{Gr}} (A_{ij}, \mc{X})} \map_{\mbf{Gr}} (B_{ij}, \mc{Y}). 
\]
To show that this has the right lifting property with respect to $\Lambda^n_i \to \Delta^n$, we simply note that $\mbf{Gr}^{\Sp^{\Sigma}}_{\mc{O}}$ has a simplicial model structure. 
\end{proof}

Because it is necessary in the sequel, we note that there is a corresponding notion of non-unital category and non-unital category monad. It will be crucial in what follows. 

\begin{defn}
Let $\mbf{T}_{\text{nu}}: \mbf{Gr} \to \mbf{Gr}$ be the \textbf{non-unital category functor}
\[
\mbf{T}_{\text{nu}} \mc{G} = \bigvee_{n=1} \mc{G}^{\otimes n}
\]
This is a monad with the obvious product $\mbf{T}_{\text{nu}} \mbf{T}_{\text{nu}} \to \mbf{T}_{\text{nu}}$ and unit given by $\mc{G} \to \mbf{T}_{\text{nu}} \mc{G}$, the inclusion of the first term. 
\end{defn}

The same conclusions follow about non-unital categories as for categories. For later use, we introduce some notation. 

\begin{nota}
We let $\cat^{\operatorname{nu}}$ denote the category of small non-unital spectral categories. 
\end{nota}

Since all categories we deal with be algebras of monads, we pause here to record a useful (though trivial) fact.

\begin{lem}\label{split_coequalizer}
Let $X \in \mc{C}[\mbf{M}]$ with structure map $\xi: \mbf{M} X \to X$ then the following is a split coequalizer
\[
\xymatrix{
\mbf{M}\mbf{M} X \ar@<-.5ex>[r]_{\mbf{M}\xi}\ar@<.5ex>[r]^{\mu} & \mbf{M} X \ar[r]_{\xi} & X
}
\]
\end{lem}

The point is that this will allow us to prove theorems for free objects, and then use something like preservation of colimits to prove for general $\mbf{T}$-algebras.

\begin{rmk}

The topological setting is more involved. We would have to verify that the Cofibration Hypothesis is satisfied. We will not be working in settings (for example, equivariant settings) where the full power of the topological context will be needed, so we will not develop this here. 
\end{rmk}

\subsection{Other Categories of Concern}

It is regrettable, but for later use we must define spectral categories with action by another spectral category. It is noted in \cite{schwede_shipley_equivalences} that there are two reasonable notions of what could be called a ``module'' for a spectral category. 

\begin{enumerate}
\item One could either look at category $\mc{M}$ equipped with maps $\mc{A} \otimes \mc{M} \to \mc{M}$ where the product is the one on graphs Defn. \ref{graph_product}. 
\item One could consider functors $F: \mc{A}^{\text{op}} \to \Sp$.
\end{enumerate}

The second case has been extensively studied, and is correspondingly useful. The first case is less well studied. We will call these ``categories with $\mc{A}$-action'' to avoid overloading the term ``module'', which will already have some semantic issues. 

\subsection{Categories with $\mc{A}$-action}

\begin{defn}
Let $\mc{A}$ be a spectral category with fixed object set $|\mc{A}|$. Then a \textbf{category with $\mc{A}$-action} is a spectral category $\mc{M}$ with object set $|\mc{A}|$ equipped with a map
\[
\mc{A} \otimes \mc{M} \to \mc{M}.
\]
\end{defn}

\begin{nota}
We let $\cat^{\mc{A}\text{-act}}_{|\mc{A}|}$ denote the category of such objects --- the morphisms are the functors $\mc{M} \to \mc{N}$ commuting with the action. 
\end{nota}

\begin{nota}
There is also a non-unital version of this, which we will denote with the cumbersome notation $\cat^{\mc{A}\text{-act}, \operatorname{nu}}_{|\mc{A}|}$
\end{nota}

\begin{rmk}
The notion ``category with $\mc{A}$-action'' could have been defined for categories without a fixed object set. However, some fiddling as in the construction in the introduction would have been required. 
\end{rmk}

We can similarly define, and will need categories with a two-sided action. 

\begin{defn}
Let $\mc{A}, \mc{B}$ be spectral categories with fixed object set $\mc{O}$. A \textbf{spectral category with $(\mc{A},\mc{B})$-action} is a spectral category equipped with a map
\[
\mc{A} \otimes \mc{M} \otimes \mc{B} \to \mc{M}
\]
\end{defn}

\begin{rmk}
This is equivalent to defining a category with action by $\mc{A}\otimes \mc{B}^{\text{op}}$. 
\end{rmk}

\begin{defn}
Let $\mbf{A} \mc{M}$ denote $\mc{A} \otimes \mc{M}$. 
\end{defn}

We note that these categories with $\mc{A}$-actions are thus $\mbf{A}$-algebras in the category of $\mbf{T}$-algebras of graphs. Thus, by \cite[II.6.1]{EKMM}, categories with $\mc{A}$-actions are algebras over the composite monad $\mbf{A} \mbf{T}$ which is
\[
\mbf{A} \mbf{T} \mc{G} = \mc{A} \otimes \left( \bigvee_{n=0} \mc{G}^{\otimes n}\right) = \bigvee_{n=0} \mc{A} \otimes \mc{G}^{\otimes n}
\]
where note that $\mc{A} \otimes \mc{G}^{\otimes 0} = \mc{A}$. 

\begin{prop}
The monad $\mbf{A} \mbf{T}$ preserves reflexive coequalizers, as does $\mbf{A}\mbf{T}_{nu}$. 
\end{prop}

\begin{proof}
It is enough to show that $\mbf{A}$ preserves reflexive coequalizers since both $\mbf{T}$ and $\mbf{T}_{\text{nu}}$ do. However, as long as the monoidal structure preserves colimits (which it does in this case) this will hold. 
\end{proof}

\begin{cor}
Spectral categories with $\mc{A}$-action and non-unital spectral categories with $\mc{A}$-action form simplicial and topological model categories. The model structure is determined by that of $\mbf{Gr}^{\text{Sp}^\Sigma}$. 
\end{cor}

\subsection{Modules}

Recall the definition of a module for a spectral category. 

\begin{defn}
Let $\mc{A}$ be a spectral category. A \textbf{module} is a (spectral) functor $F: \mc{A}^{\text{op}} \to \Sp$. The category of all such is denoted $\operatorname{Mod}_{\mc{A}}$. 
\end{defn}

\begin{rmk}
Examining the definition, this means that if $\mc{M}: \mc{A}^{\text{op}} \to \Sp$ is an $\mc{A}$-module, then $\mc{M}(a)$ is a spectrum and for the morphism spectrum $\mc{A}(a,b)$ we have a map $\mc{M}(a) \sma \mc{A}(a,b) \to \mc{M}(b)$. 
\end{rmk}

\begin{rmk}\label{yoneda}
Note that there is a spectral Yoneda embedding $\mc{A} \to \operatorname{Mod}_{\mc{A}}$ given by $a \mapsto \mc{A}(-, a)$. 
\end{rmk}

\begin{thm}[Schwede-Shipley, \cite{schwede_shipley_equivalences, schwede_shipley_stable}]
The category $\operatorname{Mod}_{\mc{A}}$ is a spectral model category where the weak equivalences are level-wise weak equivalences of (symmetric) spectra. 
\end{thm}

Furthermore, we will need the relationship between these categories as $\mc{A}$ varies. If $\mc{A} \simeq \mc{A}'$ as categories with fixed objects $|\mc{A}|$, then clearly the following holds. 

\begin{lem}[\cite{schwede_shipley_stable}]
If $F: \mc{A} \to \mc{A}'$ is a weak equivalence in $\mc{C}\text{at}^{\Sp}_{|\mc{A}|}$ then 
\[
\operatorname{Mod}_{\mc{A}} \leftrightarrows \operatorname{Mod}_{\mc{A}'}
\]
is a Quillen equivalence. 
\end{lem}

We also need to define bimodules. In order to do so we note that the category of spectral categories has a symmetric monoidal structure. If $\mc{A}$ and $\mc{B}$ are spectral categories, then $\mc{A} \sma \mc{B}$ has as objects $\operatorname{ob}\mc{A} \times \operatorname{ob} \mc{B}$ and morphism spectra
\[
(\mc{A}\sma \mc{B}) ((a,b), (a',b')) = \mc{A}(a,a') \sma \mc{B}(b,b'). 
\]
With this in place we have the following definition. 
\begin{defn}\label{bimodule}
An $(\mc{A}, \mc{B})$-bimodule is an $\mc{A} \sma \mc{B}^{\text{op}}$-module. The category of such is denoted $\operatorname{Mod}_{(\mc{A},\mc{B})}$. 
\end{defn}

\begin{rmk}
The the theorem of Schwede-Shipley cited above, this also obviously has a stable model structure. 
\end{rmk}

Although $\mc{A}$-modules and graphs with $\mc{A}$-action are not the same thing, it turns out that $\mc{A}\sma \mc{A}^{\text{op}}$-modules and graphs with $\mc{A}\sma \mc{A}^{\text{op}}$-action \textit{are} the same thing. As per the warning in \cite{schwede_shipley_equivalences}, $\mc{A}$-modules and graphs with $\mc{A}$-action cannot be the same thing, for the simple reason that graphs with $\mc{A}$-action have to be indexed by the set $|\mc{A}|\times |\mc{A}|$, and $\mc{A}$-modules are not indexed by this. However, an $\mc{A}\sma \mc{A}^{\text{op}}$-module is a functor $\mc{M}$ which for $a, b \in \mc{A}$ takes a value $\mc{M}(a, b)$ such that we have actions $\mc{A}(i, a) \sma \mc{M}(a, b) \to \mc{M}(i, b)$ and a similar action on the right. This is \textit{exactly} the definition of a graph with $\mc{A}\sma\mc{A}^{\text{op}}$-action. 

We record this remark a a lemma.

\begin{lem}
A graph with $(\mc{A},\mc{A})$-action is the same as an $(\mc{A}, \mc{A})$-bimodule. 
\end{lem}

\section{Adjunctions}

Following Basterra \cite{basterra} we will identify the left adjoint of the square zero extension functor as a composition of various mediating functors. In this section we collect the necessary adjunctions. We will define the spectral category theoretic versions of the augmentation ideal and  indecomposables and identify their adjoints. We will further show that all of these functors enjoy the necessary homotopical properties.

Before we give the definition, we need to define a non-unital version of a spectral category.

\begin{defn}
  Let $\mc{A}$ be a spectral category. Let $\ast_{|\mc{A}|}$ be the spectral category such that $\mc{A}(a, a) = S$, the sphere spectrum, and $\mc{A}(a, b) = \ast$ when $a \neq b$.
\end{defn}

\begin{defn}
  Let $\mc{A}^{\mbf{nu}}$ be denote the cofiber of the unit map $\ast_{|\mc{A}|} \to \mc{A}$. 
\end{defn}

\begin{defn}
We define a functor $I: (\cat^{(\mc{A},\mc{A})\text{-act}}_{|\mc{A}|})_{/\mc{A}} \to \cat^{(\mc{A},\mc{A})\text{-act},\mbf{nu}}_{|\mc{A}|}$ as follows. Given an object $\mc{C} \to \mc{A}$ we form the pullback
\[
\xymatrix{
I \mc{C} \ar[r]\ar[d] & \mc{C}\ar[d]\\
\mc{A}^{\text{nu}} \ar[r] & \mc{A}
}
\]
The resulting $I \mc{C}$ has an $\mc{A}$-action induced from the one on $\mc{C}$, and is a non-unital category again via the composition structure of $\mc{C}$. 
\end{defn}

\begin{defn}
We define a functor $K: \cat^{(\mc{A},\mc{A})\text{-act}, \mbf{nu}}_{|\mc{A}|} \to (\cat^{(\mc{A},\mc{A})\text{-act}}_{|\mc{A}|})_{/\mc{A}}$ as follows. Given $\mc{N}$ a non-unital spectral category with $(\mc{A},\mc{A})$-action, we let 
\[
K(\mc{N}) = \mc{A} \vee \mc{N}
\]
 where $\vee$ is the pointwise sum. We given $\mc{A} \vee \mc{N}$ a category structure by giving it a square zero multiplication. That is, we have
\begin{align*}
&(\mc{A}\vee \mc{N})(a, b) \sma (\mc{A} \vee \mc{N}) (b, c) \\
&\simeq (\mc{A}(a, b) \sma \mc{A}(b, c)) \vee (\mc{A}(a,b) \vee \mc{N}(b, c)) \vee (\mc{N}(a, b) \sma \mc{A}(b,c)) \vee (\mc{N}(a,b) \sma \mc{N}(b,c)) \\
&\to \mc{A}\vee \mc{N} (a, c)
\end{align*}
where the map is defined by composition in $\mc{A}$, the bimodule structure on $\mc{N}$ and mapping the last term to $\ast$. 

The map to $\mc{A}$ is given by projection $\pi: \mc{A} \vee \mc{N} \to \mc{A}$. 
\end{defn}

\begin{rmk}
Note that $\vee$ is most certainly \textit{not} the coproduct in the category of spectral categories. It is, however, in the category of spectral graphs. 
\end{rmk}

These functors $K$ and $I$ are adjoint. The proof mimics the proof found in \cite{basterra}. 

\begin{prop}
There is an adjunction
\[
K: \cat^{(\mc{A},\mc{A})\text{-act}, nu}_{|\mc{A}|} \leftrightarrows (\cat^{(\mc{A},\mc{A})\text{-act}}_{|\mc{A}|})_{/\mc{A}} : I
\]
that is, 
\[
(\cat^{(\mc{A},\mc{A})\text{-act}}_{|\mc{A}|})_{/\mc{A}} (K (\mc{N}), \mc{C}) \cong \cat^{(\mc{A},\mc{A})\text{-act},\mbf{nu}}_{|\mc{A}|} (\mc{N}, I\mc{C})
\]
\end{prop}
\begin{proof}
We prove it for free non-unital categories. Let $G \in \mbf{Gr}^{Sp}_{|\mc{A}|}$ and consider the free non-unital category $\mbf{T}_{\text{nu}} \mc{G}$ and suppose $\mc{C}$ is a spectral category with $\mc{A}$-action. Then 
\begin{align*}
\cat^{(\mc{A},\mc{A})\text{-act}}_{/\mc{A}} (K(\mbf{AT}_{\text{nu}} G), \mc{C}) &\cong \cat^{(\mc{A},\mc{A})\text{-act}}_{/\mc{A}} (\mbf{TA} G, \mc{C})\\
&\cong \mbf{Gr} (G, I(\mc{C}))\\
&\cong \cat^{(\mc{A},\mc{A})\text{-act},nu} (\mbf{T}_{nu} G, \mc{C})
\end{align*}

The result now follows since a general $G$ will be a split coequalizer of free objects, by Lem. \ref{split_coequalizer}. 
\end{proof}

We want the homotopical properties of this. 

\begin{prop}
The adjunction above is a Quillen adjunction and in fact a Quillen equivalence. 
\end{prop}
\begin{proof}
To show the functor is Quillen, it is enough to show that the left adjoint functor preserves cofibrations and acyclic cofibrations. To see this, we note in general that cofibrations for model structures induced by a monad $\mbf{M}$ are retracts of relative cell $\mbf{M}$-algebras. Thus, in $\cat^{(\mc{A},\mc{A})\text{-act}}_{/\mc{A}}$ the cofibrations are retracts of relative cell $\mbf{AT}$-algebras and in $\cat^{(\mc{A},\mc{A})\text{-act},nu}$ the cofibrations are retracts of relative $\mbf{AT}_{nu}$-algebras. 

The rest of the proof proceeds as in \cite{basterra}. 
\end{proof}

We shall need some more functors and adjunctions. 

\begin{defn}
We define the \textbf{indecomposables} functor 
\[
Q: \cat^{(\mc{A},\mc{A})\text{-act},\mbf{nu}}_{|\mc{A}|} \to \operatorname{Mod}_{(\mc{A},\mc{A})}
\]
 is defined as follows. Let $\mc{N}$ be a non-unital spectral category with $(\mc{A},\mc{A})$-action. Then we consider $\mc{N}$ as an $\mc{A}$-bimodule and form the following pushout in $\operatorname{Mod}_{(\mc{A},\mc{A})}$
\[
\xymatrix{
\mc{N} \otimes \mc{N}\ar[d]\ar[r] & \ast \ar[d]\\
\mc{N} \ar[r]_r & Q(\mc{N})
}
\]
\end{defn}

\begin{defn}
The functor $Z: \operatorname{Mod}_{(\mc{A},\mc{A})} \to \cat^{(\mc{A},\mc{A})\text{-act}, nu}_{|\mc{A}|}$ is defined as follows. Given a module $\mc{M}: \mc{A}\sma \mc{A}^{\text{op}} \to \Sp$ we define a spectral category $Z \mc{M}$ to have morphisms $\mc{M}(a, b)$, with composition zero and $(\mc{A},\mc{A})$-action given by the action induced from $\mc{M}$. 
\end{defn}

\begin{prop}
We have an adjunction
\[
Q: \cat^{(\mc{A},\mc{A})\text{-act},\mbf{nu}}_{|\mc{A}|} \leftrightarrows \operatorname{Mod}_{(\mc{A},\mc{A})} : Z
\]
that is
\[
\operatorname{Mod}_{(\mc{A},\mc{A})} (Q(\mc{N}), \mc{M}) \cong \cat^{\mc{A}\text{-act},nu} (\mc{N}, Z(\mc{M}))
\]
\end{prop}
\begin{proof}
Again, we suppose $\mc{N} = \mbf{AT}_{nu} G$. Then $Q (\mc{N}) = G$, so 
\[
\cat^{\mc{A}\text{-act},nu} (\mbf{T}_{\text{nu}} G, Z(\mc{M})) \cong \operatorname{Mod}_{(\mc{A},\mc{A})} (G, \mc{M}) \cong \operatorname{Mod}_{(\mc{A},\mc{A})} (Q(\mbf{T}_{nu} G), \mc{M})
\]
\end{proof}

\begin{prop}
The adjunction is Quillen. 
\end{prop}
\begin{proof}
We need only that $Z$ preserve fibrations and acyclic fibrations. However, $\operatorname{Mod}_{(\mc{A},\mc{A})}$ has a model structure that is the same as graphs with $(\mc{A},\mc{A})$-action. The result follows. 
\end{proof}

There is also a free-forgetful adjunction between $\operatorname{Mod}_{(\mc{A},\mc{A})}$ and $\cat^{(\mc{A},\mc{A})\text{-act},\mbf{nu}}_{|\mc{A}|}$. 

\begin{prop}
Let $F: \operatorname{Mod}_{(\mc{A},\mc{A})} \to \cat^{(\mc{A},\mc{A})\text{-act},\mbf{nu}}_{|\mc{A}|}$ by noting that $\operatorname{Mod}_{(\mc{A},\mc{A})}$ is a spectral graph with $\mc{A} \sma \mc{A}^{\text{op}}$-action and taking the free-non-unital category monad. 

The forgetful fuctor $U:\cat^{(\mc{A},\mc{A}),nu} \to \operatorname{Mod}_{(\mc{A},\mc{A})}$ is defined in an obvious way. 
\end{prop}

\begin{rmk}
As per the above remark, we will denote the free functor from $\operatorname{Mod}_{(\mc{A},\mc{A})}$ to $\cat^{(\mc{A},\mc{A}),nu}$ by $\mbf{T}_{nu}$. 
\end{rmk}

\begin{prop}
The free-forgetful adjunction
\[
\mbf{T}_{nu} : \operatorname{Mod}_{(\mc{A},\mc{A})} \leftrightarrows \cat^{(\mc{A},\mc{A})\text{-act},nu} : U
\]
is a Quillen adjunction. 
\end{prop}
\begin{proof}
We just need that $U$ preserve fibrations and acyclic fibrations. This follows from the fact that $\operatorname{Mod}_{(\mc{A},\mc{A})}$ is the same as the category of graphs with $\mc{A} \otimes \mc{A}$-action. 
\end{proof}

\section{(Co)homology}

With all of these adjunctions in place, we may now define the absolute and relative cotangent complexes, as well as homology and cohomology. As in other contexts, we want the cotangent complex to be left adjoint to the square-zero extension. 

\begin{defn}
Let $M$ be an $(\mc{A},\mc{A})$-module. Then define the square zero extension $\mc{A} \ltimes M$, an object of $\cat^{\Sp}_{/\mc{A}}$ (and in fact $\cat^{\Sp,|\mc{A}|}_{/\mc{A}}$) as follows. The objects are the same as $\mc{A}$. The morphism spectra re $\mc{A}(a,b) \vee \mc{M}(a, b)$ with composition given by composition in $\mc{A}$, the action of $\mc{A}$ on $\mc{M}$ and the composition in $\mc{M}$ is 0. 
\end{defn}

We now have the following string of Quillen adjunctions
\[
\xymatrix{
(\cat^{\Sp}_{|\mc{A}|})_{/\mc{A}} \ar@<.5ex>[r]^{\mc{A}\amalg - \amalg \mc{A}} & (\cat^{(\mc{A},\mc{A})\text{-act}}_{|\mc{A}|})_{/\mc{A}} \ar@<.5ex>[r]^I \ar@<.5ex>[l]& \cat^{(\mc{A},\mc{A})\text{-act}, \mbf{nu}}_{|\mc{A}|} \ar@<.5ex>[r]^{Q}\ar@<.5ex>[l]^K & \operatorname{Mod}_{(\mc{A},\mc{A})}\ar@<.5ex>[l]^Z
}
\]

Which gives us as in \cite{basterra, basterra_mandell}
\begin{align*}
\ho(\cat^{\Sp}_{|\mc{A}|})_{/\mc{A}} (\mc{C}, \mc{A}\ltimes M) &\cong \ho(\cat^{(\mc{A},\mc{A})\text{-act}}_{|\mc{A}|})_{/\mc{A}} (\mc{A} \amalg \mc{C} \amalg \mc{A}, \mc{A} \ltimes M) \\
&\cong \ho(\cat^{(\mc{A},\mc{A})\text{-act}, \mbf{nu}}_{|\mc{A}|}) (I^R (\mc{A} \amalg \mc{C}\amalg \mc{A}), M)\\
&\cong \ho(\operatorname{Mod}_{(\mc{A},\mc{A})}) (Q^L I^R (\mc{A} \amalg \mc{C} \amalg \mc{A}), M)
\end{align*}

This entitles us to the following definition. 

\begin{defn}
We define the \textbf{absolute cotangent complex of $\mc{A}$} to be 
\[
\mbf{L}_{\mc{A}} = Q^L \mc{A}^{\text{nu}}
\]

We define the \textbf{absolute topological Andre-Quillen cohomology} with coefficients in $\mc{M}$, an $\mc{A} \sma \mc{A}^{\text{op}}$-module, to be
\[
\operatorname{TDC}(\mc{A}, \mc{M}) := \pi_\ast R\operatorname{Mod}_{\mc{A}\sma\mc{A}^{\text{op}}}(\mbf{L}_{\mc{A}}, \mc{M})
\]
where we are using the spectral enrichment of $\operatorname{Mod}_{(\mc{A},\mc{A})}$. 

The \textbf{absolute topological Andre-Quillen homology} is defined to be 
\[
\operatorname{TDH} (\mc{A}, \mc{M}) := N^{cy} (\mc{A}^c, B(\mc{M} , \mc{A}, \mbf{L}_{\mc{A}}))
\]
where the $c$ superscript denotes (pointwise) cofibrant replacement. 
\end{defn}

\begin{rmk}
Note that a more usual invariant of spectral categories, topological Hochschild homology, is defined as
\[
\thh(\mc{A},\mc{M}) := N^{\text{cy}}(\mc{A}^c, \mc{M}) 
\]
or can be defined in a more canonical way, see \cite{blumberg_mandell_localization}. 
\end{rmk}

Given the above we get long exact sequences relating the above homology theories with more common ones.

\begin{prop}
Let $\mc{A}$ be a cofibrant spectral category. Then there is a fiber sequence
\[
\operatorname{TDH}(\mc{A}) \to \bigvee_{a,b} \mc{A}(a,b) \to \thh(\mc{A})
\]
\end{prop}
\begin{proof}
Consider the homotopy cofiber sequence
\[
\mc{A}^{\mbf{nu}} \otimes \mc{A}^{\mbf{nu}} \to \mc{A}^{\mbf{nu}} \to \mbf{L}_{\mc{A}}.
\]
Apply $N^{\text{cy}}(\mc{A}, -)$ and then use \ref{thh_cofiber}. 
\end{proof}

\begin{rmk}
This says that topological derivation homology and topological Hochschild homology differ from each other by the coarsest possible invariant of the spectral category --- that is, all of its morphism spaces taken together. 
\end{rmk}

\subsection{Relative Story}

We could very well define the relative cotangent complex via the procedures above, but it would only amount to to the relative cotangent complex for categories with the same objects. We would of course like to consider relative cotangent complex for categories with different objects. To this end, we need some theorems relating module structures. 

\begin{prop}
Given $F: \mc{A} \to \mc{B}$ there is an adjunction
\[
F_\ast : \operatorname{Mod}_{(\mc{A},\mc{A})} \leftrightarrows \operatorname{Mod}_{(\mc{B},\mc{B})} : F^\ast
\]
and it is Quillen. 
\end{prop}
\begin{proof}
It is clear that the right adjoint preserves fibrations and trivial fibrations. 
\end{proof}

We have the spectral version of \cite[6.2]{tabuada}

\begin{prop}
Let $F: \mc{A} \to \mc{B}$ be a spectral functor. Then, there is an induced morphism
\[
\mbf{L}_F : \mathbb{L} F_\ast (\mbf{L}_{\mc{A}}) \to \mbf{L}_B
\]
in the homotopy category of $\mc{B}$-bimodules. 
\end{prop}
\begin{proof}
It suffices to construct an induced map $\mbf{L}_{\mc{A}} \to F^\ast (\mbf{L}_{\mc{B}})$. This would live in $\operatorname{Mor}(\mbf{L}_{\mc{A}}, F^\ast (\mbf{L}_{\mc{B}}))$, which by adjunction is
\[
\operatorname{Mor}(\mc{A}, \mc{A} \ltimes (F^\ast \mbf{L}_{\mc{B}}))
\]
However, there is a natural transformation
\[
\operatorname{Mor}(\mc{B}, \mc{B} \ltimes (-)) \Longrightarrow \operatorname{Mor}(\mc{A}, \mc{A} \ltimes F^\ast (-))
\]
But there is a distinguished element in $\operatorname{Mor}(\mc{B}, \mc{B} \ltimes \mbf{L}_{\mc{B}})$, namely, the one coming from the identity by adjunction. 
\end{proof}

\begin{defn}
Let $F: \mc{A} \to \mc{B}$ be a spectral functor. Then define the \textbf{relative cotangent complex} as the homotopy cofiber
\[
\mathbb{L} F_\ast (\mbf{L}_{\mc{A}}) \to \mbf{L}_{\mc{B}} \to \mbf{L}_{\mc{B}/\mc{A}}
\]
\end{defn}

One can then quickly define relative topological derivation homology and cohomology. 

\begin{defn}
\textbf{Relative topological derivation (co)homology} with coefficients in an $\mc{B}$-bimodule $\mc{M}$ is given by 
\[
\operatorname{TDC}(\mc{B}/\mc{A}; \mc{M}) := \pi_\ast \mbf{R}\operatorname{Mod}_{(\mc{B},\mc{B})} (\mbf{L}_{\mc{B}/\mc{A}}, \mc{M})
\]

Similarly, \textbf{relative topological derivation homology} will be given by the cyclic bar construction 
\[
\operatorname{TDH}(\mc{B}/\mc{A}) := N^{\text{cy}} (\mc{B}^c; \mbf{L}_{\mc{B}/\mc{A}})
\]
\end{defn}

We can now go through the elementary properties of topological derivation (co)homology. The following are proved exactly as in \cite{tabuada}. 

\begin{prop}[Transitivity Sequence]
  Suppose we have functors
  \[
  \mc{A} \xrightarrow{F} \mc{B} \xrightarrow{G} \mc{C}.
  \]
  Then there is a homotopy cofiber sequence
\[
\mathbb{L} G_\ast (\mbf{L}_{\mc{B}/\mc{A}}) \to \mbf{L}_{\mc{C}/\mc{A}} \to \mbf{L}_{\mc{C}/\mc{B}}
\]
\end{prop}

\begin{prop}[Mayer-Vietoris]
Given a pushout square
\[
\xymatrix{
\mc{A} \ar[d]_{i}\ar[r]^F & \mc{B}\ar[d]^j\\
\mc{A}' \ar[r]_{F'} & \mc{B}'
}
\]
in $\cat_{\Sp}$ \cite{tabuada_spectral}, then the induced morphism
\[
\mathbb{L}F'_\ast (\mbf{L}_{\mc{A}'/\mc{A}}) \vee \mathbb{L}G_\ast (\mbf{L}_{\mc{B}/\mc{A}}) \to \mbf{L}_{\mc{B}'/\mc{A}}
\]
is a weak equivalence in $(\mc{B},\mc{B})$-bimodules.
\end{prop}

\subsection{Triangulated Equivalence}

First, we recall the notion of DK-equivalence, which is one notion of equivalence for spectral categories. 

\begin{defn}
Two spectral categories (objects not necessarily the same) $\mc{A}, \mc{B}$ are \textbf{Dwyer-Kan equivalent} (hereafter, DK-equivalent) if
\begin{enumerate}
\item $\pi_0 (\mc{A}) \to \pi_0 (\mc{B})$ is an equivalence of categories (i.e. is fully faithful and essentially surjective)
\item For $a, a' \in \mc{A}$, the map
\[
\mc{A}(a, a') \to \mc{B}(F(a), F(a'))
\]
is a weak equivalence of spectra. 
\end{enumerate}
\end{defn}

We want to show that the theory is invariant with respect to DK-equivalence. This will allow us to descend $\operatorname{TDH}$ to stable $(\infty,1)$-categories. First, we note the following.

\begin{rmk}
Let $F: \mc{A} \to \mc{A}'$ be a DK-equivalence of spectral categories. Then there is a weak equivalence of $\mc{A}$-modules $\mc{A} \to F^\ast \mc{A}'$. 
\end{rmk}

\begin{prop}\label{DK_equiv}
Let $F:\mc{A} \to \mc{A}'$ be a DK-equivalence of spectral categories. Then
\[
\mbf{L}_{\mc{A}} \to F^\ast \mbf{L}_{\mc{A}'} 
\]
is a weak equivalence of $(\mc{A},\mc{A})$-bimodules (with Morita model structure, see \cite{blumberg_gepner_tabuada}). 
\end{prop}
\begin{proof}
This follows from the characterization of the map and the definition of weak equivalence for modules. More, precisely, we have the following cofiber sequences and maps between them
\[
\xymatrix{
\mc{A}^{\text{nu}} \otimes \mc{A}^{\text{nu}} \ar[d]_{\simeq} \ar[r] & \mc{A}^{\text{nu}} \ar[d]_{\simeq}\ar[r] & \mbf{L}_{\mc{A}} \ar[d] \\
F^\ast ((\mc{A}')^{\text{nu}} \otimes (\mc{A}')^{\text{nu}}) \ar[r] & F^\ast ((\mc{A}')^{\text{nu}}) \ar[r] &  F^\ast \mbf{L}_{\mc{A}'} 
}
\]
The indicated maps are Morita equivalences, and so the final map is. 
\end{proof}

\begin{prop}
Let $\mc{A} \to \mc{A}'$ be a DK-equivalence of spectral categories. Then 
\[
\operatorname{TDH}(\mc{A}) \to \operatorname{TDH}(\mc{A}')
\]
is a weak equivalence of spectra. 
\end{prop}
\begin{proof}
First, we note that $\operatorname{Mod}_{(\mc{A},\mc{A})}$ and $\operatorname{Mod}_{(\mc{A}',\mc{A}')}$ are spectrally Quillen equivalent. Now, $\mbf{L}_{\mc{A}}$ and $\mbf{L}_{\mc{A}'}$ live in different categories, namely $\operatorname{Mod}_{(\mc{A},\mc{A})}$ and $\operatorname{Mod}_{(\mc{A}',\mc{A}')}$. However, by Prop.\ref{DK_equiv}, we have a weak equivalence in $\operatorname{Mod}_{(\mc{A},\mc{A})}$
\[
\mbf{L}_{\mc{A}} \to F^\ast \mbf{L}_{\mc{A}'}. 
\]
In then follows that
\[
N^{\text{cy}}(\mc{A}, \mbf{L}_{\mc{A}}) \simeq N^{\text{cy}} (\mc{A}',  \mbf{L}_{\mc{A}'})
\]
since $\thh$ preserves Morita equivalences \cite[Thm. 5.10]{blumberg_mandell_localization}, which proves the assertion. 
\end{proof}

\begin{prop}
Let $\mc{A} \to \mc{A}'$ be a DK-equivalence of spectral categories. Then 
\[
\operatorname{TDC}(\mc{A}') \to \operatorname{TDC}(\mc{A})
\]
is a weak equivalence of spectra. 
\end{prop}
\begin{proof}
We have to prove
\[
\mbf{R}\hom (\mbf{L}_{\mc{A}}, \mc{A}) \simeq \mbf{R}\hom (\mbf{L}_{\mc{A}'}, \mc{A}').
\]
However, this follows since $F^\ast$ is an enriched equivalence and $\mc{A}$ and $F^\ast \mc{A}'$ are equivalent as $\mc{A}$-modules; $\mbf{L}_{\mc{A}}$ and $F^\ast \mbf{L}_{\mc{A}'}$ are equivalent as well. 
\end{proof}

Results of \cite{blumberg_gepner_tabuada} allow us to state the following corollary. 

\begin{cor}
Topological derivation cohomology and homology descend to invariants
\[
\operatorname{TDH}: \cat^{\operatorname{Ex}}_\infty \to \Sp, \ \operatorname{TDC}: \cat^{\operatorname{Ex}}_{\infty} \to \Sp. 
\]
\end{cor}

\begin{rmk}
It follows formally from \cite{tabuada} that topological derivation (co)homology is the correct invariant to phrase obstruction problems in stable $\infty$-categories. That is, these will provide Postnikov invariants for stable $\infty$-categories. 
\end{rmk}

\section{The Stabilization of Spectral Categories}

Given a model category $\mc{C}$ which is cofibrantly generated, we can form the category of spectra of $\mc{C}$, called $\Sp(\mc{C})$, and easily endow it with a model structure coming from the cofibrantly generated model structure. 

In what follows, we stress that we will not be constructing the usual sort of category of spectra that topologists are concerned with. That is, much work on modern categories of spectra has centered on obtaining a sufficiently well-behaved smash product (i.e. one that is strictly associative). Since we care only about stabilization and not products, we will not be concerned with this, and this is why we work with just \textit{spectra} instead of, say, \textit{symmetric spectra}. 

Following \cite{basterra_mandell} we will denote suspension by $EX$ in $\mc{C}$. 

\begin{defn}
Since $X \in \mc{C}$ and $\mc{C}$ has a simplicial model structure and so is tenorsed, the object $X \sma S^1$ exists. We denote it by $EX$. 
\end{defn}

Since we will be using methods of \cite{basterra_mandell} to prove the result below, we use the definition of spectra in a general model category found therein. Before we give the definition, we need to define the ``free spectra''
\begin{defn}
Let $X \in \mc{C}$. Then $F_n X$ is a spectrum such that
\[
(F_n X)_j = 
\begin{cases}
\ast & j < n\\
E^{j-n} x & j \geq n
\end{cases}
\]
\end{defn}

\begin{defn}[Spectra in a cofibrantly generated model category]
Let $\mc{C}$ be a simplicial cofibrantly generated model category with (acyclic) cofibrations $I$ (resp. $J$). A \textbf{spectrum in $\mc{C}$} will be a sequence of objects $\{X_i\}$ together with a sequence of maps of maps $EX_{n} \to X_{n+1}$. The model structure is given by generating cofibrations and generating acyclic cofibrations
\begin{align*}
I_{\Sp} &= \{F f: f \in I\}\\
J_{\Sp} &= \{F g: g \in I\} \cup \{\lambda_{n;f} :f \in I\}
\end{align*}
where $\lambda_{n;f}$ is given as follows.  Given $f: S \to T$ we define
\begin{align*}
S_{n;f} &= F_n S \amalg_{F_{n+1} (ES)} F_{n+1} (ES \otimes I) \amalg_{F_{n+1} (ES)} F_{n} (ET)\\
T_{n;f} &= F_n T \amalg_{F_{n+1} (ET)} F_{n+1} (ET \otimes I)
\end{align*}
with $\lambda_{n;f}$ induced by $f$. 
\end{defn}

\begin{thm}
The categories $\Sp(\cat^{\text{Sp}}_{/\mc{A}})$, $\Sp(\cat^{|\mc{A}|, \Sp}_{/\mc{A}})$, $\Sp(\cat^{(\mc{A},\mc{A})\text{-act}}_{/\mc{A}})$ all have cofibrantly generated model structures. 
\end{thm}
\begin{proof}
Follows from the previous definition/theorem
\end{proof}

\begin{thm}[Main Theorem]
We have adjunctions
\[
\xymatrix{
\operatorname{Mod}_{(\mc{A},\mc{A})} \ar@/^/[r]^{\Sigma^\infty} &  \Sp(\operatorname{Mod}_{(\mc{A},\mc{A})}) \ar@/^/[r]^{\mbf{T}_{\text{nu}}}  \ar[l]^{(-)_0} & \Sp (\cat^{(\mc{A},\mc{A})\text{-act},u}) \ar@/^/[l]^U \ar@/^/[r]^{K}& \Sp (\cat^{(\mc{A},\mc{A})\text{-act}}_{/\mc{A}}) \ar@/^/[l]^{I}
}
\]
and they all induce Quillen equivalences. 
\end{thm}
\begin{proof}
This is proved as in \cite{basterra_mandell}. The first and last adjunctions are Quillen equivlances: $\operatorname{Mod}_{(\mc{A},\mc{A})}$ is a stable model category, so is Quillen equivalent to its category of spectra and the adjunction $(K, I)$ is an equivalence before passing to spectra. Thus, one only needs to prove that the free-forgetful adjunction is a a Quillen equivalence. The only minor difference is the treatment of \cite[Lem. 8.9]{basterra_mandell}. However, with the monad we're using, the map in their paper is automatically a weak equivalence. 

\end{proof}

Again, the following is proved as in \cite{basterra_mandell}. 

\begin{lem}
The $Q, Z$ adjunction
\[
\xymatrix{
\Sp(\operatorname{Mod}_{(\mc{A},\mc{A})}) \ar@/^/[r] & \Sp (\cat^{(\mc{A},\mc{A})\text{-act},nu}) \ar@/^/[l]
}
\]
is a Quillen equivalence. 
\end{lem}

\begin{thm}
Let $M \in \operatorname{Mod}_{(\mc{A},\mc{A})}$. Then the equivalence between $\operatorname{Mod}_{(\mc{A},\mc{A})}$ and $\Sp(\cat^{(\mc{A},\mc{A})\text{-act}}_{/\mc{A}})$ is given by 
\[
M \mapsto \{\mc{A} \vee Z \Sigma^n \mc{M}\}_n
\]
\end{thm}

As in \cite{basterra_mandell}, this also entitles us to statements about cohomology theories of spectra categories. 

\begin{thm}[Cohomology Theories] 
Category of cohomology theories on $\cat_{/\mc{A}}$ is equivalent to the homotopy category of $\mc{A}$-bimodules. 

The category of homology theories on $\cat_{\mc{A}}$ is equivalent to the category of homology theories on $\mc{A}$-bimodules. 
\end{thm}

\section{Appendix: (Enriched) Bar Constructions}

\begin{defn}
Let $\mc{C}$ be a category enriched in spectra. We define the bar construction, or nerve, $N (\mc{C})$ to be the geometric realization of the simplicial spectrum 
\[
N_q (\mc{C}) = \bigvee_{c_0, \dots, c_q} \mc{C}(c_0, c_1) \sma \mc{C}(c_1, c_2) \sma \cdots \sma \mc{C}(c_{q-1}, c_q)
\]
with face maps given by composition in the category and degeneracies given by units. 
\end{defn}

\begin{defn}
Let $\mc{C}$ be a spectral category and $\mc{M}$ a $(\mc{C},\mc{C})$-bimodule. We define the \textbf{cyclic bar construction} to be the geometric realization of the simplicial set
\[
N^{\text{cy}}_q (\mc{C};\mc{M}) = \bigvee_{c_0, \dots, c_{q+1}} \mc{C}(c_0, c_1) \sma \mc{C}(c_1, c_2) \sma \cdots \sma \mc{C}(c_{q-1}, c_q) \sma \mc{M}(c_q, c_0)
\]
With the exception of $d_0$ and $d_q$ the face maps are given by composition in the category $\mc{C}$. The other face maps are given by the action of $\mc{C}$ on $\mc{M}$. Unit maps in $\mc{C}$ give the degeneracies. 
\end{defn}

Another similar construction of spectral categories the two-sided bar construction. This will be useful in various places. 

\begin{defn}[Two-Sided Bar Construction]\label{two_sided_bar_construction}
Let $\mc{M}$ be a left $\mc{A}$-module and $\mc{N}$ a right $\mc{A}$-module, then the \textbf{two-sided bar construction} is the realization of the simplicial spectrum with $q$-simplices
\[
B_q(\mc{M},\mc{A},\mc{N}) = \bigvee_{a_0, a_1, \dots, a_q} \mc{M}(a_0) \sma  \mc{A}(a_0, a_1) \sma \mc{A}(a_1, a_2) \sma \dots \sma \mc{A}(a_{q-1}, a_q) \sma \mc{N}(a_q)
\]
\end{defn}

The following two sided bar lemma is found in \cite{blumberg_mandell_localization}. Its proof is the same as the same lemma in \cite{may_goils}

\begin{lem}\cite[Lem. 6.3]{blumberg_mandell_localization}\label{two_sided_bar}
Let $\mc{A}$ be a small spectral category and consider $\mc{M}$ a right $\mc{A}$-module and $\mc{N}$ a left $\mc{A}$-module. Given $a \in \mc{A}$, consider also the image under the spectral Yoneda embedding, $a \mapsto \mc{A}(-,a)$. Then
\[
B_\bullet (\mc{M}, \mc{A}, \mc{A}(a,-)) \xrightarrow{\simeq} \mc{M}(a)
\]
and
\[
B_\bullet (\mc{A}(-,a); \mc{A}; \mc{N}) \to \mc{N}(a)
\]
are simplicial homotopy equivalences. 
\end{lem}

We will also need the following theorem from \cite{blumberg_mandell_localization}. 

\begin{prop}\cite[Prop. 3.7]{blumberg_mandell_localization}\label{thh_cofiber}
Cofiber sequences of $\mc{A}$-modules $\mc{M} \to \mc{M}' \to \mc{M}''$ induce cofiber sequences on $\thh$:
\[
\thh(\mc{A}, \mc{M}) \to \thh(\mc{A},\mc{M}') \to \thh(\mc{A},\mc{M}'')
\]
\end{prop}

\begin{prop}
Let $\mc{A} \otimes \mc{A}$ be as in the body of the text. Then 
\[
N^{\text{cy}} (\mc{A}, \mc{A} \otimes \mc{A}) \simeq \bigvee_{a,b} \mc{A}(a,b)
\]
\end{prop}
\begin{proof}
We have
\begin{align*}
N^{\text{cy}}_q (\mc{A}; \mc{A}\otimes \mc{A}) &= \bigvee_{o_0, \dots, o_q} \mc{A}(o_{q-1}, o_q) \sma \cdots \sma \mc{A}(o_0, o_1) \sma (\mc{A} \otimes \mc{A}) (o_q, o_o)\\
&= \bigvee_{o_0, \dots, o_q} \mc{A} (o_{q-1}, o_q) \sma \cdots \sma \mc{A} (o_0, o_1) \sma \bigvee_{o_{q+1}} \mc{A} (o_{q}, o_{q+1}) \sma \mc{A} (o_{q+1}, o_0)\\
&= \bigvee_{o_o, \dots, o_{q+1}} \mc{A}(o_q, o_{q+1}) \sma \mc{A}(o_{q-1},o_q) \sma \cdots \sma \mc{A}(o_o,o_1)
\end{align*}

Now, noting that 
\[
\bigvee_{o_0, \dots, o_{q+1}} \mc{A}(o_0, o_1), \dots, \mc{A}(o_q, o_{q+1}) = \bigvee_{o_o,o_{q+1}}B_\bullet (\mc{A}(-,o_{q+1}),\mc{A}, \mc{A}(o_o,-))
\]
we can invoke \ref{two_sided_bar} to obtain a simplicial homotopy
\[
\bigvee_{o_o,o_{q+1}} B(\mc{A}(-,o_{q+1}), \mc{A}, \mc{A}(o_0, -)) \to \bigvee_{o_o, o_{q+1}} \mc{A}(o_0, o_{q+1})
\]
whence the lemma. 
\end{proof}

\bibliographystyle{amsplain}
\bibliography{ssc}

\end{document}